\documentclass[11pt]{amsart}
\usepackage{amsmath,amsthm,amsfonts,amscd,amssymb,eucal,latexsym,mathrsfs,stmaryrd,enumitem,bm}
\usepackage[all]{xy}
\usepackage{mathtools}
\usepackage{enumitem}  
\usepackage{comment}
\usepackage{xcolor}
\usepackage[initials, alphabetic]{amsrefs}
\DeclarePairedDelimiter\ceil{\lceil}{\rceil}

\newtheorem{theorem}{Theorem}[section]

\newtheorem{lemma}[theorem]{Lemma}
\newtheorem{claim}[theorem]{Claim}
\newtheorem{proposition}[theorem]{Proposition}

\theoremstyle{definition}
\newtheorem{definition}[theorem]{Definition}

\theoremstyle{plain}
\newcounter{theoremintro}
\newtheorem{theoremi}[theoremintro]{Theorem}

\newcommand{\cgr}{{\mathrm{CGR}}}
\newcommand{\lex}{{\mathrm{lex}}}

\newcommand{\cE}{{\mathcal E}}

\newcommand{\Zb}{{\mathbb Z}}

\newcommand{\Nb}{{\mathbb N}}

\newcommand{\brak}[1]{\left(#1\right)}

\newcommand{\asdim}{\text{asdim}_B}
\newcommand{\emb}{\hookrightarrow}

\numberwithin{equation}{section}

\renewcommand{\phi}{\varphi}

\allowdisplaybreaks

\begin{document}

\title[Hyperfiniteness and Borel asymptotic dimension of boundary actions]{Hyperfiniteness and Borel asymptotic dimension of boundary actions of hyperbolic groups}

\thanks{The authors were funded by the Deutsche Forschungsgemeinschaft (DFG, German Research Foundation) under Germany’s Excellence Strategy – EXC 2044 – 390685587, Mathematics Münster – Dynamics – Geometry – Structure, the Deutsche Forschungsgemeinschaft (DFG, German Research Foundation) – Project-ID 427320536 – SFB 1442, and ERC Advanced Grant 834267 - AMAREC.
}

\author{Petr Naryshkin}
\address{Petr Naryshkin,
Mathematisches Institut,
WWU M{\"u}nster, 
Einsteinstrasse\ 62, 
48149 M{\"u}nster, Germany}
\email{pnaryshk@uni-muenster.de}

\author{Andrea Vaccaro}
\address{Andrea Vaccaro, Mathematisches Institut,
WWU M{\"u}nster, 
Einsteinstrasse\ 62, 
48149 M{\"u}nster, Germany}
\email{avaccaro@uni-muenster.de}


\begin{abstract}
We give a new short proof of the theorem due to Marquis and Sabok, which states that the orbit equivalence relation induced by the action of a finitely generated hyperbolic group on its Gromov boundary is hyperfinite. Our methods permit moreover to show that every such action has finite Borel asymptotic dimension.
\end{abstract}


\maketitle

\section{Introduction}

The study of countable Borel equivalence relations is a prominent subject in descriptive set theory, whose interactions with dynamics have been thriving ever since
Feldman and Moore showed in \cite{FM} that every countable Borel equivalence relation is equal to the orbit equivalence relation of a Borel action
by a countable group. Due to this characterisation, the problem of understanding which Borel actions generate hyperfinite equivalence relations
has attracted considerable attention in the past decades.

This line of research began with Weiss' seminal work \cite{weiss}, where he proved that any Borel action of $\Zb$ gives rise to a hyperfinite orbit equivalence relation, and asked whether this holds for all amenable groups.
His question was
followed by more than 30 years of research during which larger and larger classes of amenable groups have been shown to generate hyperfinite
equivalence relations (\cite{JKL, GJ, SS}), culminating in the recent paper \cite{ConJacMarSewTuc20}. The main novelty of \cite{ConJacMarSewTuc20} is the
notion of \emph{Borel asymptotic dimension} (Definition \ref{def:Basdim}),
a Borel analogue of the homonymous large-scale invariant introduced by Gromov
in \cite{gromov}, which provides a sufficient condition to verify hyperfiniteness. More precisely, if $X$ is a standard Borel space
with extended metric $\rho$ (i.e. which also allows value $\infty$)
and $(X,\rho)$ has finite Borel asymptotic dimension, then the finite-distance
equivalence relation is hyperfinite (\cite[Theorem 1.7]{ConJacMarSewTuc20}).
Thanks to this innovative approach, the authors of \cite{ConJacMarSewTuc20} were able to
provide a positive answer to Weiss' question for the largest class of amenable groups to date.

Turning to non-amenable groups, Weiss' question is known to have negative answer. Indeed, every non-amenable group admits a Borel action whose orbit equivalence relation is not hyperfinite (\cite[Corollary 1.8]{JKL}). Nevertheless, it is still worthwhile to study specific natural actions of non-amenable groups, particularly when the orbit equivalence relation is amenable.
In this direction, \cite[Corollary 8.2]{DJK} shows that the tail equivalence relation on $S^\Nb$ is hyperfinite whenever $S$ is countable. An immediate
consequence of this is that the orbit equivalence relation of the action of a finitely generated free group $\mathbb{F}$ on its Gromov boundary
$\partial \mathbb{F}$ is hyperfinite. Indeed, $\partial \mathbb{F}$ naturally
embeds into $S^\Nb$, where $S$ is a symmetric set of free generators,
and with this identification the orbit equivalence relation on $\partial \mathbb{F}$ coincides with the tail equivalence on $S^\Nb$.  

Building on this result, in \cite{cubulated} it is proved that
the orbit equivalence relation of the boundary action $G \curvearrowright \partial G$ is hyperfinite for all cubulated hyperbolic groups $G$, and this was later successfully extended in \cite{MarSab} to all finitely generated hyperbolic groups. More recently, a further generalization of these methods was
carried in \cite{relative}, where a similar theorem is obtained for
relatively hyperbolic groups. More results on hyperfiniteness of orbit equivalence relations of boundary actions can
be found in \cite{PrSa}.

The key difference between a general hyperbolic group, equipped with a finite generating set $S$, and a free group is that for the former there is no canonical embedding $\partial G \emb S^\Nb$ that maps the orbit equivalence to a well-understood equivalence relation of the image. To fill this gap, the papers \cite{cubulated, MarSab, relative} develop sophisticated geometric techniques to study families of geodesic ray bundles, and to adapt the methods used in \cite{DJK}
to general boundary actions of hyperbolic groups.

The first main contribution of this paper is a new short proof of the main result
of \cite{MarSab}, stated explicitly below, which avoids most
of the geometric machinery employed in \cite{MarSab} and provides a direct
connection between the orbit equivalence relation of $G\curvearrowright \partial G$ and the tail equivalence on $S^\Nb$.

\begin{theoremi}[{\cite[Theorem A]{MarSab}}]\label{theorem:HF}
    The orbit equivalence relation induced by the action of a finitely generated hyperbolic group on its Gromov boundary is hyperfinite.
\end{theoremi}

Our strategy to prove Theorem \ref{theorem:HF} starts with fixing
some linear order on a finite generating set $S$ of $G$. This induces a 
well-defined Borel function mapping each $\eta \in \partial G$ to the unique geodesic ray converging to it which is minimal (or rather whose \emph{type} is minimal, see \S \ref{S.Borel}) in the lexicographical order on $S^\Nb$. That provides an embedding $\Phi \colon \partial G \emb S^\Nb$ (Definition \ref{def:Gamma_eta}). While it might not be the case that the orbit equivalence is mapped to the tail equivalence under this map, we show in Lemma \ref{lemma:bound} that, thanks to hyperbolicity, the former contains only finitely many classes of the latter. The result then follows from \cite[Proposition 1.3 (vii)]{JKL}.

Another advantage of this approach is that our methods can be easily adapted to the framework introduced in \cite{ConJacMarSewTuc20}. Indeed, if $\rho_s$
is the extended metric on $S^\Nb$ induced by the left shift $s \colon S^\Nb \to S^\Nb$, then \cite[Lemma 8.3]{ConJacMarSewTuc20} shows that $(S^\Nb, \rho_s)$
has Borel asymptotic dimension 1. Our technical result
Lemma \ref{lemma:bound} shows that we can essentially pull back this information
via the aforementioned embedding $\Phi \colon \partial G \to S^\Nb$ to
$\partial G$, and obtain the following result.

\begin{theoremi}\label{theorem:asdim}
    The action of a finitely generated hyperbolic group on its Gromov boundary has finite Borel asymptotic dimension. 
\end{theoremi}

The paper is organized as follows. In \S \ref{S.pre} we recall some preliminaries on hyperbolic groups and prove a few basic facts about geodesic rays. In \S \ref{S.Borel} we isolate a specific Borel
uniformizing function from $\partial G$ to $S^\Nb$. In \S \ref{S.main} we recall the definitions of hyperfiniteness and Borel asymptotic dimension and give the proofs of Theorems \ref{theorem:HF} and \ref{theorem:asdim}.

\section{Preliminaries on hyperboilc groups} \label{S.pre}
We start with some essential preliminaries and notation on graphs, geodesic rays and hyperbolicity. We refer to \cite{MarSab}*{\S 2}
for a more detailed, self-contained introduction on these topics (for a
more throughout presentation of this subject we recommend instead \cite{BH}).

A \emph{graph} $X = (V, \cE)$ is a pair composed of a set of \emph{vertices} $V$ and a (symmetric) set of \emph{edges}
$\cE \subseteq V \times V$.
Two vertices $x,y \in V$ are \emph{adjacent}
if $(x,y) \in \cE$, and a \emph{path} $\Gamma$ in $X$ is a (possibly infinite) sequence of vertices $(x_n)_{n \in I}$ indexed by an interval $I \subseteq \Nb$ such that $(x_n, x_{n+1}) \in \cE$
for every $n \in I$. We shall sometimes denote $x_n$ by $\Gamma(n)$.
The \emph{graph metric} on $V$ is defined by setting $d(x,y)$ as the smallest possible length of a path
joining $x$ and $y$, and it has value $\infty$ if no such path exists. A \emph{geodesic path} is a path $(x_n)_{n \in I}$
such that $d (x_n, x_m) = | n - m |$ for every $n,m \in I$.
Two infinite geodesic paths $\Gamma_0, \Gamma_1$ are \emph{asymptotic} if their Hausdorff
distance $d_H(\Gamma_0,\Gamma_1)$ is finite. The \emph{visual boundary} $\partial X$ of the graph $X$ is the set of equivalence classes
of asymptotic infinite geodesic paths. Given $x \in V$ and $\eta \in \partial X$, a \emph{combinatorial geodesic ray} (or simply a \emph{CGR})
from $x$ to $\eta$ is a geodesic path $\Gamma = (x_n)_{n \in \Nb}$ whose equivalence class in $\partial X$ is $\eta$ and such that $x_0 = x$.
In this case, we say that $\Gamma$ \emph{starts} at $x_0$ and \emph{converges} to $\eta$.
Let $\cgr(x, \eta)$ denote the set of all CGR starting from $x$ and
converging to $\eta$.

Given $\delta > 0$, a locally finite graph $X = (V,\cE)$ (i.e. such that
every vertex has finitely many adjacent vertices)
is \emph{$\delta$-hyperbolic} if
for every geodesic triangle (a triangle whose sides are geodesic paths) $\Delta$ of $X$, 
each side of $\Delta$ is contained in the $\delta$-neighborhood of the other
two sides. In this case $\partial X$ can be equipped with a compact metrizable topology as follows. Fix a base point $x \in V$.
We say that a sequence $\eta_n \in \partial X$ converges to some $\eta \in \partial X$
if and only if there are $\Gamma_n \in \cgr(x,\eta_n)$ such that every subsequence of $(\Gamma_n)_{n \in \Nb}$ has a subsequence
converging to some $\Gamma \in \cgr(x, \eta)$ in the space $V^\Nb$ equipped with the product topology resulting from the discrete topology on $V$. The set $\partial X$ equipped with such topology, which does not depend on the choice of
the base point $x$, is called the \emph{Gromov boundary} of $X$ (see \cite{BH}*{\S III.3}).

Let $G$ be a finitely generated group and let $S$ be a finite symmetric generating set of $G$. The \emph{Cayley graph}
$\mathrm{Cay}(G,S)$ is the locally finite graph whose vertex set is $G$ and where $g,h \in G$ are adjacent if and only if $g^{-1}h \in S$.
The group $G$ is \emph{hyperbolic} is there exists $\delta > 0$ such that $\mathrm{Cay}(G,S)$ is $\delta$-hyperbolic. It is well-known that being hyperbolic is invariant with respect of quasi-isometry, and therefore it does not depend on the choice of $S$.
We denote $\partial \mathrm{Cay}(G,S)$ simply as $\partial G$, since also its topology does not depend on $S$ (see \cite{BH}*{Theorem 3.9}). The action
of $G$ on its Cayley graph naturally induces an action on combinatorial
geodesic rays which preserves asymptoticity. This in turn results
in a continuous amenable action $G\curvearrowright \partial G$ on its Gromov boundary (see e.g. \cite[Theorem B.3.4]{ADR}).

We isolate a couple of elementary properties of hyperbolic graphs which will play an important role in our arguments in \S \ref{S.main}.
\begin{lemma}[{\cite{BH}*{Lemma III.3.3}}]
\label{lemma:TubularLemma}
Let $X$ be a $\delta$-hyperbolic graph and let $\eta \in \partial X$. Suppose that $\Gamma_0 = \brak{x_n}_{n \in \Nb}$ and $\Gamma_1 = \brak{y_n}_{n \in \Nb}$ are CGR converging to $\eta$. 
    \begin{enumerate}[label=(\arabic*)]
\item \label{lemma:item1} If $x_0 = y_0$ then
\[
d(x_n, y_n) \le 2\delta, \text{ for every } n \in \Nb.
\]
\item In general, there are
\[
0 \le T_0, T_1 \le 2(d(x_0, y_0) + 3\delta + 2)
\]
such that 
\[
d(x_{T_0 + n}, y_{T_1+n}) \le 5\delta, \text{ for every } n \in \Nb.
\]
\end{enumerate}
\end{lemma}
\begin{proof}

     Item \ref{item:1} is exactly \cite{BH}*{Lemma III.3.3.(1)}.
    
     For item \ref{item:2}, the existence of $T_0, T_1 \ge 0$ such that $d (x_{T_0 + n}, y_{T_1+n}) \le 5\delta$ for every $n \in \Nb$ is granted by \cite{BH}*{Lemma III.3.3.(2)}. The upper bound $T_0, T_1 \le 2(d (x_0, y_0) + 3\delta + 2)$ does not appear explicitly in the lemma, but can be deduced from the proof therein.
    
    More precisely, let $c_n$ be a sequence of geodesic paths from
    $x_0$ to $y_n$. As $X$ is locally finite,
    there exists a subsequence of $(c_n)_{n \in \Nb}$ converging to some CGR $\Gamma$ whose starting point is $x_0$. Given $n \in \Nb$, consider the
    geodesic triangle whose edges are $c_0$, $c_n$ and the initial segment $(y_0,\dots, y_n)$ of $\Gamma_1$. By $\delta$-hyperbolicity, the path $c_n$ is contained in a $\delta$-neighborhood of $c_0 \cup \Gamma_1$. We claim that the vertex $c_n(m)$ belongs to a $\delta$-neighborhood of
    $\Gamma_1$ for every $m \ge d (x_0, y_0)+\delta +1$ and every $n$ for which $c_n(m)$ is defined. Indeed, if this were not the case,
    then $c_n(m)$ would belong to a $\delta$-neighborhood of $c_0$, hence
    giving
    \[
    d (c_n(m),x_0) \le d (x_0, y_0) + \delta < m,
    \]
    and contradicting the fact that
    $c_n$ is a geodesic path from $x_0$ to $c_n(m)$. This in particular
    implies that $\Gamma$ converges to $\eta$. A similar argument shows
    that if $R_0 = d (x_0, y_0) + \ceil{\delta} +1$ and $m \ge 2(d (x_0, y_0)+\delta +1)$, then
    $d (c_n(R_0), y_m) > \delta$.
    
    Summarizing, this means
    that there exists $R_1 \le 2(d (x_0, y_0) + \delta + 1)$ such that
    $d (\Gamma(R_0), y_{R_1}) \le \delta$. Notice that, given $n \in \Nb$, the
    vertex $\Gamma(R_0+n)$ could be in the $\delta$-neighborhood of a $y_{\ell}$ for $\ell < R_1$,
    but by the triangle inequality there exists a minimum $k_0 \in \Nb$ such that
    for every $n \ge k_0$ there is $n' \in \Nb$ satisfying
    \[
    d (\Gamma(R_0 + n), y_{R_1+n'}) \le \delta.
    \]
    Again by the triangle inequality it follows that $|n - n'| \le 2 \delta$,
    hence
    \[
    d (\Gamma(R_0 + n), y_{R_1+n}) \le 3\delta \text{ for every }n \ge k_0.
    \]
    Since both $\Gamma_0$ and $\Gamma$ belong to $\cgr(x_0, \eta)$, then
    $d (\Gamma(n),x_n) \le 2 \delta$ for every $n \in \Nb$
    by part \ref{lemma:item1}, thus we get
    \[
     d (x_{R_0+k_0+n}, y_{R_1+k_0+n}) \le 5\delta \text{ for every }n \in \Nb.
    \]
    
    We claim that $T_i = R_i + k_0$ satisfy the required bound.
    If $k_0 = 0$ there is nothing to prove since both $R_0$ and $R_1$
    are smaller than $2(d (x_0, y_0) + \delta + 1)$. If instead $k_0 >0$, let $m \in \Nb$
    be such that
    \[
    d (\Gamma(R_0 + k_0), y_{R_1+m}) \le \delta.
    \]
    The triangle inequality implies $|k_0 - m| \le 2 \delta$ and,
    as $\Gamma(R_0+ k_0-1)$ is in the $\delta$-neighborhood of some $y_s$ for $s <
    R_1$, it also implies $m \le 2\delta +1$. We can thus conclude that
    $k_0\le 4\delta +1$, and that $T_0$ and $T_1$ satisfy the required upper bound also when $k_0 > 0$.
\end{proof}

Given a finite path $\Gamma_0=(x_n)_{n < k}$ and a path $\Gamma_1 = (y_n)_{n \in I}$ in a graph $X = (V,\mathcal{E})$ such that $(x_{k-1},y_0) \in \mathcal{E}$, we denote by
$\Gamma_0 \cdot \Gamma_1$ the path which starts with $\Gamma_0$ and then continues along $\Gamma_1$.
\begin{lemma}
\label{lemma:SubstitutionInCGR}
Let $X$ be a graph,
let $\Gamma = \brak{x_n}_{n \in \Nb}$ be a CGR, let $n_1 > n_0 \in \Nb$ and let $\Gamma_{01}$ be any geodesic path from $x_{n_0}$ to $x_{n_1}$. Then $\brak{x_n}_{n<n_0}\cdot\Gamma_{01}\cdot\brak{x_n}_{n > n_1}$ is a CGR.
\end{lemma}

\begin{proof}
By assumption $\Gamma$ is a CGR, hence $d (x_0, x_{n-1}) = n-1$ for
every $n \in \Nb$.
Let $\brak{y_n}_{n \in \Nb} = \brak{x_n}_{n< n_1}\cdot\Gamma_{01}\cdot\brak{x_n}_{n > n_2}$ and suppose that it is not a CGR. Then there exists $N \in \Nb$ such that for all $n > N$ we have $d (y_0, y_{n-1}) < n-1$. However, $y_0 = x_0$ and $\Gamma_{01}$ is a geodesic path between $x_{n_0}$
and $x_{n_1}$,
hence $y_n = x_n$ for every $n > n_1$. This contradicts the fact that $\brak{x_n}_{n \in \Nb}$ is a CGR.
\end{proof}

\section{Borel Uniformizing Functions} \label{S.Borel}
Given a hyperbolic group $G$ with finite symmetric generating set $S$,
in this section we define a specific Borel injection which associates to each $\eta \in \partial G$ a sequence in $S^\Nb$. Let $\Gamma = (x_n)_{n \in \Nb} \in G^\Nb$ be a CGR in $\text{Cay}(G,S)$. Its \emph{type} $\text{typ}(\Gamma)$ is defined as
\[
\text{typ}(\Gamma) = (x_n^{-1} x_{n+1})_{n \in \Nb} \in S^\Nb.
\]
We use an analogous notation also for finite paths.

In what follows we consider $\partial G$ with its topology as Gromov boundary, while $S^\Nb$ and $G^\Nb$ are endowed with the product topology resulting from the
discrete topology on $S$ and $G$ respectively.

\begin{proposition}[\cite{MarSab}*{Lemma 5.1}]\label{prop:compact}
    The set $\cgr(e, \eta)$ is compact in $G^\Nb$.
\end{proposition}

\begin{proof}
    As $\text{Cay}(G,S)$ is locally finite, any sequence $\brak{\Gamma_n}_{n\in \Nb} \subseteq \cgr(e, \eta)$ has a subsequence converging to some geodesic ray $\Gamma$. It follows from Lemma \ref{lemma:TubularLemma} that $d(\Gamma(m), \Gamma_0(m)) < 2\delta$ for all $m \in \Nb$ which shows that $\Gamma$ converges to $\eta$.
\end{proof}

\begin{definition}\label{def:Gamma_eta}
    Let $G$ be a hyperbolic group with finite symmetric generating set $S$. Fix a total order on $S$ and let $\le_\lex$ be the resulting lexicographical order on $S^\Nb$. Consider the map
\begin{align*}
    \Phi \colon \partial G &\to S^\Nb \\
\eta &\mapsto \sigma_\eta \nonumber
\end{align*}
where $\sigma_\eta$ is the unique element in $S^\Nb$ such that
\[
\sigma_\eta = \min\nolimits_{\le_\lex}\{\text{typ}(\Gamma) \mid \Gamma \in \cgr(e,\eta) \},
\]
with $e$ being the identity element in $G$. The existence of such minimal element follows from Proposition \ref{prop:compact}.
\end{definition}

\begin{proposition} \label{prop:Borel}
    The function $\Phi \colon \partial G \to S^\Nb$ in Definition
    \ref{def:Gamma_eta} is a Borel injection.
\end{proposition}
\begin{proof}
    Consider the set
    \[
    A = \{ (\eta,  \Gamma) \in \partial G \times G^\Nb \mid \Gamma \in \text{CGR}
    (e, \eta) \}.
    \]
    The proof of \cite{MarSab}*{Claim 6.3} shows that $A$ is closed in $\partial G \times G^\Nb$.  
    Moreover, given $\eta \in \partial G$, the section
    \[
    A_\eta = \{ \Gamma \in G^\Nb \mid (\eta, \Gamma) \in A \}
    \]
    is precisely $\cgr(e,\eta)$ which is a compact subset of $G^\Nb$ by Proposition \ref{prop:compact}. Then by \cite{Kech}*{Theorem 28.8}
    the map
    \begin{align*}
        \Theta \colon \partial G &\to K(G^\Nb) \\
        \eta &\mapsto \cgr(e,\eta)
    \end{align*}
    is Borel, where $K(G^\Nb)$ is the space of compact
    subsets if $G^\Nb$ with the Vietoris topology (see \cite[\S I.4.F]{Kech}). Consider next the function
    \begin{align*}
        \Psi \colon K(G^\Nb) &\to S^\Nb \\
        F &\mapsto \min\nolimits_{\le_\lex} \{ \text{typ}(\Gamma) \mid \Gamma \in F \}.
    \end{align*}
    We claim that also $\Psi$ is Borel. Indeed,
    let $U$ be the basic open set in $S^\Nb$ consisting of all sequences starting with some
$(s_0, \dots, s_{k-1}) \in S^k$. Let $U'$ be the open set in $G^\Nb$
of all elements starting with $(x_0, \dots, x_{k-1})$, where $x_n = s_0 s_1\dots s_{n}$ for every $n < k$. Given $\bar y = (y_0, \dots, y_{k-1}) \in G^k$ denote the corresponding basic open set in $G^\Nb$ as
\[
V_{\bar y} = \{ (z_n)_{n \in \Nb} \in G^\Nb \mid z_n = y_n \text{ for all } n < k\}.
\]
Let $C = \{ \bar y \in G^k \mid (s_0, \dots, s_{k-1}) \le_\lex \text{typ}(\bar y) \}$ and consider the open set
\[
V=\bigcup_{ \bar y \in C} V_{\bar y}.
\]

The preimage $\Psi^{-1}(U)$ consists of all $F \in K(G^\Nb)$
which intersect $U'$ and which are contained in $V$, and it is therefore
and open set in the Vietoris topology of $K(G^\Nb)$. We can thus conclude that $\Phi = \Psi \circ \Theta$ is Borel.

Injectivity of $\Phi$ follows since any two CGR which start from the same point and have the same type must be equal.
\end{proof}

\section{Hyperfiniteness and Borel Asymptotic Dimension} \label{S.main}

This section is devoted to the proof of Theorem \ref{thm:main}, which covers both Theorem \ref{theorem:HF} and \ref{theorem:asdim}. We begin with
some preliminaries on Borel equivalence relations and by recalling
the notion of Borel asymptotic dimension as introduced in \cite{ConJacMarSewTuc20}.

An \emph{extended metric} $\rho$ on a set $X$ is a metric on $X$ which can also take
value $\infty$. For $x \in X$ and $r > 0$, let $B_\rho(x,r)$ denote
the $\rho$-ball of radius $r$ centered in $x$.
Let $E_\rho$ be the equivalence relation of finite distance between points
\[
E_\rho = \{(x,y) \in X \times X \mid \rho(x,y) < \infty \}.
\]

Given a standard Borel space $X$, an equivalence relation $E$ on $X$ is
\emph{Borel} if $E$ is a Borel subset of $X \times X$. A Borel equivalence
relation $E$ is \emph{finite} (respectively \emph{countable}) if all its equivalence classes are finite (respectively countable),
and it is \emph{hyperfinite} if it can be written as a union
of an increasing sequence of finite Borel equivalence relations.

A \emph{Borel extended metric space} $(X,\rho)$ is a standard Borel space
paired with a Borel extended metric $\rho$. In this case $E_\rho$ is a Borel
equivalence relation on $X$. An equivalence relation $E$ on a Borel extended metric space $(X,\rho)$ is \emph{uniformly bounded} if there exists a uniform finite bound
on the $\rho$-diameter of the $E$-classes.
A graph $X=(V,\cE)$
is a \emph{Borel graph} if $V$ is a standard Borel space and the edge relation
$\cE$ is a Borel subset of $V \times V$. Note that the graph metric endows $V$
with a structure of Borel extended metric space.

We will mainly be concerned with two types of Borel graphs.
The first one is
the \emph{Schreier graph} of a Borel action of a finitely generated group $G$
on a standard Borel space $X$. Given a finite symmetric generating set $S$
of $G$, the Schreier graph of the action $G \curvearrowright X$
has $X$ as vertex set, and two points $x,y \in X$ are adjacent
if and only if there is $g \in S$ such that $g x = y$. We remark that all vertices of such graph have degree at most $|S|$, hence for every $r > 0$
all balls of radius $r$ (in the graph metric) have size at most $|S|^r$.

The \emph{orbit
equivalence relation} $E_G^X$ on $X$ is defined as
\[
x E_G^X y \iff \exists g \in G \text{ s.t. } g x = y, \text{ for } x, y \in X.
\]
Note that the orbit equivalence relation is the same as the finite-distance equivalence relation induced by the graph metric on $X$.

The second kind of Borel graphs we focus on arises from Borel
functions. Let $X$ be a standard Borel space and let $f \colon X \to X$
be a Borel function. We let $\mathcal{G}_f$ be the Borel graph whose
vertex set is $X$, and where $x,y \in X$ are adjacent if and only if
$f(x) = y$ or $f(y) = x$. We denote the graph metric on $\mathcal{G}_f$ by $\rho_f$. The typical example we shall consider is
the left shift on $X = S^\Nb$ for a finite set $S$
\begin{align} \label{eq:fshift}
    s \colon S^\Nb &\to S^\Nb \\
    (x_{n})_{n \in \Nb} &\mapsto (x_{n+1})_{n \in \Nb} \nonumber
\end{align}
In this case the finite-distance equivalence relation $E_{\rho_s}$
corresponds to the \emph{tail equivalence} $R$ on sequences in $S$, that
is for $\sigma_0 = (x_n)_{n \in \Nb}, \sigma_1 = (y_n)_{n \in \Nb} \in S^\Nb$
\[
\sigma_0 R \sigma_1 \iff \exists T_0, T_1 \in \Nb \text{ s.t. } x_{T_0 + n} = y_{T_1+n}, \text{ for every } n \in\Nb.
\]

Finally, we recall the notion of Borel asymptotic dimension.
\begin{definition}[{\cite[Definition 3.2]{ConJacMarSewTuc20}}] \label{def:Basdim}
    Let $(X,\rho)$ be a Borel extended metric space such that $E_\rho$
    is countable. The \emph{Borel
    asymptotic dimension} $\asdim(X,\rho)$ of $X$ is the smallest $d \in \Nb$
    such that for every $r > 0$ there exists a 
    uniformly bounded Borel equivalence relation $E$ on $X$ such that for
    every $x \in X$ the ball $B_\rho(x,r)$ intersects at most $d+1$ classes of $E$,
    and it is $\infty$ if no such $d$ exists.
\end{definition}

If $G\curvearrowright X$ is a Borel action, then we denote the Borel
    asymptotic dimension of the corresponding Schreier graph simply as
    $\asdim(G\curvearrowright X)$ (note that it does not depend on the choice of the generating set $S$ of $G$ inducing the metric on $X$ by \cite[Lemma 2.2]{ConJacMarSewTuc20}).

The main goal of this section is proving the following theorem.

\begin{theorem} \label{thm:main}
Let $G$ be a hyperbolic group with finite symmetric generating set $S$.
\begin{enumerate}[label=(\arabic*)]
    \item \label{item:1} The orbit equivalence relation $E^{\partial G}_G$ induced by the boundary action $G \curvearrowright \partial G$ is hyperfinite.
    \item
    \label{item:2} Let $\delta > 0$ be such that $\text{Cay}(G,S)$ is
    $\delta$-hyperbolic and let $M_{5\delta} \in \Nb$ be a bound on the size of all balls of radius $5 \delta$ in $\text{Cay}(G,S)$.
    Then $\asdim(G \curvearrowright \partial G) \le 2 M_{5\delta}^2 -1$.
\end{enumerate}
\end{theorem}

The main motivation behind Definition \ref{def:Basdim} is \cite[Theorem 1.7]{ConJacMarSewTuc20}, stating that
if a Borel extended metric space $(X,\rho)$ has finite Borel asymptotic dimension, then the equivalence relation $E_\rho$ is hyperfinite. In view of this, it would be enough to prove
item \ref{item:2} to also obtain part \ref{item:1}. 
We point out, however, that a small modification of the methods employed
to prove the second statement of Theorem \ref{thm:main} can be used
to show that $E_G^{\partial G}$ is hyperfinite, without having to resort
to the theory developed in \cite{ConJacMarSewTuc20}. We will include this more direct proof of the first part of Theorem \ref{thm:main} for the reader's convenience.

Let us first consider Theorem \ref{thm:main} in the case where $G = \mathbb{F}_2$, the free group with two generators $a,b$. Let $S = \{a,b,a^{-1}, b^{-1} \}$ and let $s \colon S^\Nb \to S^\Nb$ be the left shift on $S^\Nb$.
Each $\eta \in \partial \mathbb{F}_2$
has a unique representative in $\cgr(e,\eta)$, and this gives a Borel injection $\Phi: \eta \mapsto \sigma_\eta$ from $\partial \mathbb{F}_2$ to $S^\Nb$ such that
\begin{equation} \label{eq:shift}
\text{either} \ \sigma_{g\eta} =
s(\sigma_\eta) \ \text{or} \ \sigma_{\eta} =
s(\sigma_{g\eta}) \text{ for every } \eta \in \partial \mathbb{F}_2 \text{ and } g \in S.
\end{equation}
It follows that $\Phi$ maps the Shreier graph metric on $\partial \mathbb{F}_2$ to the metric $\rho_s$ on $S^\Nb$. Thanks to this, the analysis of the orbit equivalence relation on $\partial\mathbb{F}_2$
is equivalent to that of the tail equivalence relation on $S^\Nb$, hence item \ref{item:1} of Theorem \ref{thm:main} directly follows from
\cite[Corollary 8.2]{DJK} and item \ref{item:2} from \cite[Lemma 8.3]{ConJacMarSewTuc20}. Our strategy to prove Theorem \ref{thm:main} is to use the Borel injection $\Phi$ from Definition \ref{def:Gamma_eta} to have a similar reduction to the tail equivalence relation
on $S^\Nb$. The situation is however more chaotic for a general hyperbolic group, since we cannot expect for an equality like the one in \eqref{eq:shift} to hold.
Nevertheless, hyperbolicity grants a sufficient control of the tail equivalence classes of $\sigma_{g\eta}$, as $g$ varies in a given bounded ball around the identity of $G$,
to fill this gap.

This will be made precise in the following lemma, for which we introduce some
notation. Let $G$ be a hyperbolic group with finite symmetric generating set $S$,
let $\delta > 0$ be such that $\text{Cay}(G,S)$ is
    $\delta$-hyperbolic and let $M_{5\delta} \in \Nb$ be an upper bound for the cardinality of the balls of radius $5 \delta$ in $\text{Cay}(G,S)$. 

\begin{lemma} \label{lemma:bound}
Let $G$ be a hyperbolic group with finite symmetric generating set $S$.
Fix $\eta \in \partial G$,
    $r > 0$, let $\rho$ be the graph metric on $\partial G$ induced by the action $G \curvearrowright \partial G$, and let $\rho_s$ be the extended
    metric on $S^\Nb$ induced by the left shift $s \colon S^\Nb \to S^\Nb$. Finally, let $\{\sigma_\zeta \}_{\zeta \in B_\rho(\eta, r)} \subseteq S^\Nb$ be the family of types of CGR given in Definition \ref{def:Gamma_eta}.
    Then $\{\sigma_\zeta \}_{\zeta \in B_\rho(\eta, r)}$ can be partitioned into at most $(M_{5\delta})^2$ sets, each one with $\rho_s$-diameter smaller than $8(r + 3\delta + 2)$.
    In particular, the set $\{\sigma_\zeta \}_{\zeta \in B_\rho(\eta, r)}$ meets
    at most $(M_{5\delta})^2$ tail equivalence classes of $S^\Nb$.
\end{lemma}

\begin{proof}
Given $\zeta \in \partial G$, let $\Gamma_\zeta \in \cgr(e,\zeta)$ be the unique CGR starting from the identity $e \in G$ such that
$\text{typ}(\Gamma_\zeta) = \sigma_\zeta$. Fix $\zeta \in B_\rho(\eta, r)$ and let $g_\zeta \in G$ be such that $d(e,g_\zeta) < r$ and $g_\zeta\zeta = \eta$. Set $\Gamma_\eta = (x_n)_{ n\in \Nb}$, $\Gamma_\zeta = (z_n)_{n \in \Nb}$ and let $g_\zeta \Gamma_\zeta = (y_n)_{n \in \Nb} \in G^\Nb$ be the sequence where $y_n = g_\zeta z_n$ for every $n \in \Nb$. Then
$g_\zeta \Gamma_\zeta\in \cgr(g_\zeta,\eta)$.

Set $T = 2(r + 3\delta + 2)$. By Lemma \ref{lemma:TubularLemma}, there is $0 \le T_\zeta \le 2T$ such that 
\[
d\brak{x_{T + n}, y_{T_\zeta + n}} \le 5\delta, \text{ for every } n \in \Nb.
\]
Note that $g_\zeta \Gamma_\zeta$ is the unique ray in $\cgr(g_\zeta,\eta)$
such that
\begin{equation} \label{eq:minimum}
\text{typ}(g_\zeta \Gamma_\zeta) = \min\nolimits_{\le_\lex}\{\text{typ}(\Gamma) \mid \Gamma \in \cgr(g_\zeta,\eta) \},
\end{equation}
since $\sigma_\zeta$ is the $\le_\lex$-minimum of $\text{typ}(\cgr(e,\zeta))$ and shifting by $g_\zeta$ gives a bijection
between $\cgr(e,\zeta)$ and $\cgr(g_\zeta,\eta)$ that preserves the type.

Fix $N \in \Nb$. By Lemma \ref{lemma:SubstitutionInCGR} and \eqref{eq:minimum}
the path $p = \brak{y_{T_\zeta+n}}_{n=0}^N$ is the unique geodesic between $y_{T_\zeta}$ and $y_{T_\zeta + N}$ such that
\[
\text{typ}(p) = \min\nolimits_{\le_\lex}\{q \in S^{N} \mid
q = \text{typ}(r) \text{ for some geodesic path } r \text{ from }
y_{T_\zeta} \text{ to } y_{T_\zeta+N} \}.
\]

Summarizing, the partial segment of $g_\zeta \Gamma_\zeta$ between the
coordinates $T_\zeta$ to $T_\zeta + N$ only depends on the initial point $y_{T_\zeta} \in B_{d}(x_{T}, 5\delta)$ and on $y_{T_\zeta + N} \in B_{d}(x_{T+N}, 5\delta)$. Since this is true
for every $N \in \Nb$, there are at most $(M_{5\delta})^2$ choices for
$\brak{y_{T_\zeta+n}}_{n\in \Nb}$ as $\zeta$ varies in $B_\rho(\eta, r)$.
Since $\text{typ}(g\Gamma_\zeta) =
\text{typ}(\Gamma_\zeta) = \sigma_\zeta$ for every $\zeta \in \partial G$ and every $g \in G$, the conclusion of the lemma follows.
\end{proof}

We are now ready to prove the main result.

\begin{proof}[Proof of Theorem \ref{thm:main}]
     To prove \ref{item:1}, let $s$ be the left shift on $S^\Nb$ and let $R$ denote the tail equivalence relation on $S^\Nb$.
    Consider the map $\Phi \colon \partial G \to S^\Nb$ from Definition \ref{def:Gamma_eta}, and let
    $R_0$ be the preimage of $R$ on $\partial G$ under $\Phi$:
    \[
    \eta R_0 \zeta \iff \sigma_\eta R \sigma_\zeta, \text{ for every } \eta,\zeta \in \partial G.
    \]

    Note that the equivalence relation $R$ is hyperfinite by \cite[Corollary 8.2]{DJK}, and thus $R_0$ is a Borel hyperfinite equivalence relation
    since $\Phi$ is a Borel injection by Proposition \ref{prop:Borel}. Let $R_1 = R_0 \cap E^{\partial G}_G$, which by \cite[Proposition 1.3.(i)]{JKL}
    is again hyperfinite (this step is actually redundant,
    as we show in Claim \ref{claim:bound}).
    Let $M_{5\delta} \in \Nb$ be an upper bound for the size of the balls of radius $5 \delta$ in $\text{Cay}(G,S)$.
    We claim that each $E^{\partial G}_G$-class is composed of
    at most $(M_{5\delta})^2$ equivalence
    classes of $R_1$. This will grant that $E^{\partial G}_G$ is hyperfinite by \cite[Proposition 1.3.(vii)]{JKL}.

    The claim is an immediate consequence of Lemma \ref{lemma:bound}, since if
    $\eta_0, \dots, \eta_{(M_{5\delta})^2}$ are elements in
    $\partial G$ in the same $E^{\partial G}_G$-class, then there exists $N > 0$ such that $\{\eta_k \}_{k \le (M_{5\delta})^2} \subseteq B_\rho(\eta_0, N)$,
    hence there are $i<j \le ({M_{5\delta}})^2$ such that $\sigma_{\eta_i} R \sigma_{\eta_j}$, and therefore $\eta_i R_1 \eta_j$.

    To prove \ref{item:2}, fix $r > 0$ and set $T = 8(r + 3\delta + 2)$. Consider $S^\Nb$ with the graph metric
    $\rho_s$ induced by the shift $s$. By \cite[Lemma 8.3]{ConJacMarSewTuc20}
    the space $(S^\Nb,\rho_s)$ has Borel asymptotic dimension at most 1. Hence
    there exists a uniformly bounded equivalence relation $ E$
    such that for every $x \in S^\Nb$ the ball $B_{\rho_s}(x, T)$, intersects at most
    2 equivalence classes of $ E$.
    
    Let $E_0$ be the preimage of $E$ via $\Phi$ from Definition \ref{def:Gamma_eta}, that is for $\eta, \zeta \in \partial G$
    we define
    \[
    \eta E_0 \zeta \iff \sigma_\eta  E \sigma_\zeta.
    \]
    Let $\rho$ be the graph metric on $\partial G$
    induced by the Schreier graph of the boundary action $G \curvearrowright
    \partial G$. The equivalence relation $E_0$ is Borel since $\Phi$ is Borel (Proposition \ref{prop:Borel}) and 
    it is uniformly bounded since $E$ is and because of the following claim.
    
    \begin{claim} \label{claim:bound}
    Let $\eta, \zeta \in \partial G$ Then
    \[ 
    \rho(\eta,\zeta) \le \rho_s(\sigma_\eta, \sigma_\zeta).
    \]
    \end{claim}
    \begin{proof}
    Suppose that there is $M\in \Nb$ such that $\rho_s(\sigma_\eta,\sigma_\zeta) \le M$.
    There exist thus $n_0, n_1 \in \Nb$ with $n_0 + n_1 \le M$
    and such that $s^{n_0}(\sigma_\eta) = s^{n_1}(\sigma_\zeta) = \bar s$. Let $s_1, \dots, s_{n_0}, t_1, \dots, t_{n_1}
    \in S$ be such that 
    \begin{equation} \label{eq:tail}
    \sigma_\eta = (s_1,\dots, s_{n_0})^\smallfrown \bar s \text{ and }\sigma_\zeta = (t_1, \dots, t_{n_1})^\smallfrown \bar s.
    \end{equation}
    Let $\Gamma_\eta\in \cgr(e,\eta)$ and $\Gamma_\zeta \in \cgr(e,\zeta)$ be the unique CGR starting from $e$ such that $\text{typ}(\Gamma_\eta) = \sigma_\eta$ and $\text{typ}(\Gamma_\zeta)
    = \sigma_\zeta$. Set $g = (s_1s_2 \dots s_{n_0})^{-1}$ and $h = (h_1 h_2 \dots h_{n_1})^{-1}$. Consider finally
    the CGR $g\Gamma_\eta$, whose $n$-th entry is equal to
    $g\Gamma_\eta(n)$, and $h\Gamma_\zeta$, whose $n$-th entry is equal to $h\Gamma_\zeta(n)$. It follows from \eqref{eq:tail}, and from the definition of $g$ and $h$,
    that $g\Gamma_\eta$ and $h\Gamma_\zeta$ are tail equivalent in $G^\Nb$, hence they converge to the same $\theta \in \partial G$. In particular, this shows that
    $g\eta = h\zeta = \theta$, and therefore $\rho(\eta, \zeta)
    \le M$.
    \end{proof}
    
    We claim that, given $\eta \in \partial G$, the ball $B_\rho(\eta, r)$ meets
    at most $2(M_{5\delta})^2$ classes of $E_0$.
    Take $\eta_0, \dots, \eta_{2(M_{5\delta})^2} \in B_\rho(\eta, r)$.
    By Lemma \ref{lemma:bound} there are $k_0<k_1< k_2 \le 2(M_{5\delta})^2$ such that
    $\{\sigma_{\eta_{k_i}} \}_{i=0,1,2}$ are pairwise tail equivalent,
    and moreover
    \[
    \rho_s(\sigma_{\eta_{k_i}} ,\sigma_{\eta_{k_j}}) \le T,  \text{ for every } i,j = 0,1,2.
    \]
    By the choice of $E$, there are distinct $i, j \in \{0, 1, 2\}$ such that $\sigma_{\eta_{k_i}} E \sigma_{\eta_{k_j}}$,
    and therefore $\eta_{k_i} E_0 \eta_{k_j}$. We can thus conclude that $B_\rho(\eta, r)$ intersects at most $2(M_{5\delta})^2$ classes of $E_0$.
\end{proof}

\bibliography{aux/main.bib}
\end{document}